\documentclass[a4paper,10pt]{article}

\usepackage{amsmath,amsthm,dsfont}
\usepackage{txfonts}
\usepackage{amssymb}
\usepackage{cite}
\usepackage{pstricks,pst-all}
\usepackage[latin1]{inputenc}
\usepackage{subfigure}
\usepackage[T1]{fontenc}
\usepackage[nottoc]{tocbibind}
\usepackage{multirow}
\usepackage{listings}
\usepackage{lmodern}
\usepackage{float}
\usepackage{color}
\usepackage{mathtools}
\usepackage{graphicx}
\usepackage[english]{babel}

\allowdisplaybreaks[1]

\begin{document}
\title{\Large\textbf{Blow-up of a stable stochastic differential equation}}
\author{Matti Leimbach, Michael Scheutzow\\
Technische Universit\"at Berlin}

\newcommand{\Addresses}{{% additional braces for segregating \footnotesize
  \bigskip
  \footnotesize
  \noindent
  M.~Leimbach, \textsc{Technische Universit\"at, Berlin, MA 7-5, Str. des 17.Juni 135, 10623 Berlin, Germany},
  \texttt{leimbach@math.tu-berlin.de}

  \medskip
  \noindent
  M.~Scheutzow, \textsc{Technische Universit\"at, Berlin, MA 7-5, Str. des 17.Juni 135, 10623 Berlin, Germany},
  \texttt{ms@math.tu-berlin.de}

}}

\date{}
\maketitle

\newtheoremstyle{dotless} % Name
                        {8pt}    % Space above
                        {14pt}    % Space below
                        {}         % Body font
                        {}         % Indent amount
                        {\bfseries}% Theorem head font
                        {}        % Punctuation after theorem head
                        {\newline} % Space after theorem head
                        {}
\theoremstyle{dotless}

\newtheorem{definition}{Definition}[section]
\newtheorem{example}[definition]{Example}
\newtheorem{theorem}[definition]{Theorem}
\newtheorem{proof+}{Proof}
\newtheorem{thesis}[definition]{Lemma}
\newtheorem{corollary}[definition]{Corollary}
\newtheorem{comment}[definition]{Remark}

\newenvironment{proof++}{\textbf{Beweis}\newline}{} %ohne Nummer

\newcommand{\Prob}[2]{{\mathbb{P}_{#2}\left(#1\right) }}
\newcommand{\PR}[2]{{\mathbb{P}_{#2}\bigg( #1 \bigg) }}
\newcommand{\EqSplit}[2]{{\begin{equation}\begin{split}
\label{#2}
#1
\text{.}
\end{split}
\end{equation}}}

%-----------------------------------------------------------------------------------

\pagestyle{myheadings}
\markright{\leftmark}

\begin{abstract}
We examine a 2-dimensional ODE which exhibits explosion in finite time. Considered as an SDE with additive white noise, it is known to be complete - in the sense that for each initial 
condition there is almost surely no explosion. Furthermore, the associated Markov process even admits an invariant probability measure. 
On the other hand, as we will show, the corresponding local stochastic flow will almost surely not be strongly complete, i.e.~there exist (random) initial conditions for which the 
solutions explode in finite time. 
\end{abstract}

\newcommand{\Wone}[1]{{W^{(1)}_{#1}}}
\newcommand{\Wtwo}[1]{{W^{(2)}_{#1}}}

\section{Introduction}

Consider the complex-valued It\^{o}-type stochastic differential equation (SDE)
\begin{equation}
\label{eq:driving}
\mathrm{d}Z_t =\left( Z_t^n + F(Z_t)\right) \mathrm{d}t+\sigma \mathrm{d}B_t \text{,}
\end{equation}
where $n\geq 2$, $\sigma \geq 0$, $F \in \mathcal{O}(\vert z \vert^{n-1})$ as $\vert z \vert \to \infty$ is locally Lipschitz and 
$B=\Wone{}+ i\Wtwo{}$ is a complex Brownian motion on a filtered probability space $\left(\Omega,
\mathcal{F},(\mathcal{F}_t),\mathbb{P}\right)$ satisfying the usual conditions.\\
Under the additional assumption that $F$ is a polynomial of $z$ and $\overline{z}$ of degree at most $n-1$ it is known 
(see \cite{HerzogMattingly1} and \cite{HerzogMattingly2}) that for every fixed initial condition $(X_0,Y_0)=(x_0,y_0)$ the one-point 
motion, i.e.\ the process which solves this equation and starts in $(x_0,y_0)$, exhibits non-explosion almost surely if $\sigma>0$ and, 
moreover, the associated Markov process admits a (unique) invariant probability measure.  
This is a remarkable fact, since there is explosion in finite time for some initial conditions in the deterministic case 
(i.e. $\sigma=0$). This is obvious in the particular case $F=0$ (take an initial condition on the positive real line) and will follow 
from our main result for general $F$. 
{Turning an explosive ODE into a non-explosive SDE with an invariant distribution by adding noise is often called noise-induced stability and was also studied in \cite{Scheutzow93} and more recently in \cite{Herzog11},\cite{Kolba11}.}\\
Now, we would like to know if the noise induces an even stronger kind of stability, namely the existence of a random attractor. In this paper, we show 
that the corresponding local stochastic flow will explode (or {\em blow up}) almost surely and therefore there cannot be a random attractor (for the definition and basic properties of 
random attractors, see \cite{CF94}).
SDEs which have a unique global solution for each initial condition are called {\em complete}. Since the local stochastic flow associated to 
(\ref{eq:driving}) explodes, it is -- by definition -- not strongly complete. So far there are only few examples which are known to be complete but not strongly complete, see for instance \cite{Elworthy78},\cite{LiScheutzow11}.

\section{Transformation into Cartesian coordinates}
For our purpose it is convenient to transform equation (\ref{eq:driving}) into Cartesian coordinates. The rest of this paper deals only with equation (\ref{eq:cartesian_driving}) 
below.\\
Denote the real and imaginary part of $F$ by $F_1$ and $F_2$, i.e. $F=F_1+iF_2$. Further, there are functions $\hat{F_1},\hat{F_2}\colon \mathbb{R}^2 \to \mathbb{R}$, such that $F_j(x+iy)= \hat{F_j}(x,y)$, $j=1,2$. If we rewrite $Z_t=X_t+iY_t$, SDE (\ref{eq:driving}) is equivalent to
\begin{equation}
\label{eq:cartesian_driving}
\begin{split}
\mathrm{d}X_t &=\left(\sum_{j=0}^{\lfloor\frac{n}{2}\rfloor }(-1)^j \binom{n}{2j} X_t^{n-2j}Y_t^{2j}+\hat{F_1}(X_t,Y_t)\right) \mathrm{d}t+\sigma \mathrm{d}\Wone{t} \text{,}\\
\mathrm{d}Y_t &=\left(\sum_{j=0}^{\lfloor\frac{n-1}{2}\rfloor }(-1)^j \binom{n}{2j+1} X_t^{n-2j-1}Y_t^{2j+1}+\hat{F_2}(X_t,Y_t)\right) \mathrm{d}t+\sigma \mathrm{d}\Wtwo{t} \text{.}
\end{split}
\end{equation}
Abbreviate 
\begin{equation*}
\begin{split}
b_1(x,y)&\coloneqq\sum_{j=0}^{\lfloor\frac{n}{2}\rfloor }(-1)^j \binom{n}{2j} x^{n-2j}y^{2j} \text{,}\\
b_2(x,y)&\coloneqq\sum_{j=0}^{\lfloor\frac{n-1}{2}\rfloor }(-1)^j \binom{n}{2j+1} x^{n-2j-1}y^{2j+1} \text{.}
\end{split}
\end{equation*}
At first sight these drift terms look quite unhandy, but the following lemma yields convenient expressions.
\begin{thesis}
For $x>0$, $y\in \mathbb{R}$ we have
\begin{equation*}
\begin{split}
b_1(x,y)&=\left(x^2+y^2\right)^{\frac{n}{2}}\cos\left(n\arctan\left(\frac{y}{x}\right)\right)  \text{,}\\
b_2(x,y)&=\left(x^2+y^2\right)^{\frac{n}{2}}\sin\left(n\arctan\left(\frac{y}{x}\right)\right) \text{.}
\end{split}
\end{equation*}
\end{thesis}
\begin{proof}
Write $z$ in Cartesian and polar coordinates, i.e. $z=x+iy=re^{i\phi}$. For $x>0$ polar coordinates can be expressed in terms of cartesian coordinates via $r=\sqrt{x^2+y^2}$, $\phi=\arctan(y/x)$. Therefore,
\begin{equation*}
\begin{split}
z^n &= (x+iy)^n= \sum_{j=0}^n \binom{n}{j}x^{n-j}(iy)^j\\& =\sum_{j=0}^{\lfloor\frac{n}{2}\rfloor }(-1)^j \binom{n}{2j} x^{n-2j}y^{2j} +i\sum_{j=0}^{\lfloor\frac{n-1}{2}\rfloor }(-1)^j \binom{n}{2j+1} x^{n-2j-1}y^{2j+1}\text{,}\\
z^n &= r^n e^{ni\phi}=r^n\cos(n\phi)+ir^n\sin(n\phi)\\
&= \left(x^2+y^2\right)^{\frac{n}{2}}\cos\left(n\arctan\left(\frac{y}{x}\right)\right)  +i \left(x^2+y^2\right)^{\frac{n}{2}}\sin\left(n\arctan\left(\frac{y}{x}\right)\right) 
\text{.}
\end{split}
\end{equation*}
The lemma follows by comparing the real and imaginary parts of both expressions.
\end{proof}
\section{Defining the problem and main result}
First, we introduce local stochastic flows on $\mathbb{R}^d$, $d \ge 1$.
\begin{definition}
Let $\mathfrak{e}(s,x),\,s \geq 0,\, x \in \mathbb{R}^d$ be a random field with values in $(s,\infty)$, such that $\mathfrak{e}(s,x)$ is lower semicontinuous in $s$ and $x$.
Set $\mathbb{D}_{s,t}(\omega)\coloneqq \{x\in\mathbb{R}^d\colon \mathfrak{e}(s,x,\omega)>t\}$ and 
let $\phi_{s,t}(x,\omega), x\in\mathbb{R}^d, 0\leq s\leq t< \mathfrak{e}(s,x)$ be a continuous $\mathbb{R}^d$-valued random field defined on the random domain of parameters $(s,t,x)$ for which $x \in \mathbb{D}_{s,t}(\omega)$. Denote the range of $\phi_{s,t}(\cdot,\omega)$ on $\mathbb{D}_{s,t}(\omega)$ by $\mathbb{R}_{s,t}(\omega)$. $\phi$ (or $\phi_{s,t}$) is called a {\em local stochastic flow}, if for almost all $\omega\in\Omega$
\begin{itemize}
\item[i)]  $\phi_{s,s}(\cdot,\omega) =\mathrm{Id}_{\mathbb{R}^d}$ for all $s \geq 0$,
%\item[ii)] $(s,t,x) \mapsto \phi_{s,t}(x,\omega)$ is continuous for all $0\leq s\leq t$, $x\in\mathbb{R}^d$ such that $x\in\mathbb{D}_{s,t}(\omega)$,
\item[ii)] $\phi_{s,t}(\cdot,\omega)\colon \mathbb{D}_{s,t}(\omega)\to\mathbb{R}_{s,t}(\omega)$ is a homeomorphism for all $0\leq s<t$ and the inverse is continuous in $(s,t,x)$,
\item[iii)]  $\phi_{s,u}(\cdot,\omega)=\phi_{t,u}(\phi_{s,t}(\cdot,\omega),\omega)$ holds on $\mathbb{D}_{s,u}(\omega)$ for all  $0\leq s\leq t\leq u$
\end{itemize}
holds true.\\
A local stochastic flow is called {\em stochastic flow} if for all $0\leq s \leq t$ and $\omega\in \Omega$ $\mathbb{D}_{s,t}(\omega)=\mathbb{R}_{s,t}(\omega)=\mathbb{R}^d$.
\end{definition}

According to \cite[Theorem 4.7.1]{Kunita97} there exists a local stochastic flow $\phi_{s,t}(x,\omega), x\in\mathbb{R}^2, 0\leq s\leq t< \mathfrak{e}(s,x)$, which is the maximal solution to equation (\ref{eq:driving}) starting at time $s$ in $x$, where $\mathfrak{e}(s,x)$ is the explosion time.\\
In the following, we write $\phi_{t}$ instead of $\phi_{0,t}$ and denote the $i$th component of $\phi_t$ by $\phi_t^{(i)}, i=1,2$. We use $\phi_t^{(1)}(z)$ and $X_t$ respectively $\phi_t^{(2)}(z)$ and $Y_t$ interchangeably, whenever the initial condition $z$ is not of importance or clear from the context.\\
Our main result is the explosion (or {\em blow up} or {\em lack of strong completeness}) of the local stochastic flow $\phi$.
\begin{theorem}
\label{thm:main}
\em
Let $\phi$ be the local stochastic flow associated to (\ref{eq:cartesian_driving}), then there exists $T \in (0,\infty)$ such that
\begin{equation*}
\label{eq:maintheorem}
\lim_{x_0\to\infty}\mathbb{P}\left(\sup_{z\in \mathfrak{I}}\sup_{t\leq T}\phi_{t}^{(1)}(z)=\infty\right)=1 \text{,}
\end{equation*}
where the initial set is given by $\mathfrak{I}\coloneqq\{x_0\}\times[-\tan\left(\frac{\pi}{2n}\right) x_0,\tan\left(\frac{\pi}{2n}\right)x_0]$.
\end{theorem}
\begin{comment}
The theorem shows that we have almost sure blow-up:
\[
\mathbb{P}\left(\exists z\in\mathbb{R}^2\colon\sup_{t\leq T}\phi_{t}^{(1)}(z)=\infty\right)\geq\lim_{x_0\to\infty}\mathbb{P}\left(\sup_{z\in \mathfrak{I}}\sup_{t\leq T}\phi_{t}^{(1)}(z)=\infty\right)=1
\text{.}
\]
\end{comment}
\section{Heuristic idea}
\label{sec:heuristic}
For the rest of this paper fix $\alpha \in \left(0,\tan\left(\frac{\pi}{2n}\right)\right)$, and define the cone
\[
\mathcal{C} \coloneqq \{(x,y)\in\mathbb{R}^2 \colon x\geq x^*, \vert y \vert \leq \alpha x \}
\text{,}
\]
where we will choose $x^*>0$ sufficiently large later on (depending only on $n$ and $F$).\\
We know that for every initial condition in $\mathcal{C}$, the solution of the SDE will almost surely eventually leave $\mathcal{C}$. Some trajectories leave this region via the upper boundary and some via the lower boundary.
Due to the continuity of the map $z \mapsto \phi_t(z)$, one may hope to be able to show that there will be (random) initial conditions in between these two kinds of points 
for which the trajectories will actually remain inside $\mathcal{C}$ forever. In the following section we will see that if such trajectories exist, then they will explode
within  time $T$ (which is small provided the initial condition has a large $x$-component) provided the noise in the $x$-direction is not too large up to time $T$.

It then remains to show that there actually exist trajectories which stay inside  $\mathcal{C}$ forever (until they blow up).  
Let us sketch the idea of the proof in case {$\Wone{}\equiv 0$}: Figure 1 shows the image of the set of initial conditions $\{(x_0,y),\,|y| \le \tan\left(\frac{\pi}{2n}\right)x_0\}$ under the map $\phi_t$ 
for some $x_0>x^*>0$ and some $t>0$. The idea of the proof is to show that, for large $x_0$, it is very unlikely that any trajectory whose $y$-coordinate happens to be above level 
$\alpha x_0/2$ at some time will hit the level $y=\alpha x_0/4$ before leaving the cone  $\mathcal{C}$ through its upper boundary (Lemma \ref{lemma:bounds}).  This will then 
allow us to show the existence of points which stay inside $\mathcal{C}$ forever (until explosion).

\begin{figure}[H]
\centering\psset{unit=0.8cm}
\begin{pspicture}(0,0)(10,11)
%\psgrid(0,0)(10,10)

\psline[linecolor=blue,linestyle=solid](1,5)(8,10)
\rput[bl](8.1,9.9){\color{blue}$y=\alpha x$}
\psline[linecolor=blue,linestyle=solid]{-}(1,5)(8,0)
\rput[bl](8.1,-0.1){\color{blue}$y=-\alpha x$}

%\psline[linecolor=magenta]{-}(1.685,5.5)(10,5.5)
%\psline[linecolor=magenta]{-}(1.685,4.5)(10,4.5)

\psline{->}(0,5)(10,5)
\psline{->}(1,0)(1,10)

%\psline{}(0.9,5.5)(1.1,5.5)
%\rput(0.8,5.5){$1$}
%\psline{}(0.9,4.5)(1.1,4.5)
%\rput(0.65,4.5){$-1$}

\rput(1,10.25){$y$}
\rput(10.2,5){$x$}

\psline[linecolor=blue]{-}(4.5,7.5)(4.5,2.5)
\rput[bl](4.35,7.7){\color{blue}$x^{*}$}

\psline[linecolor=black,linestyle=dashed](5.9,9)(5.9,1)
\rput[bl](5.65,9.1){$x_{0}$}

\psline[linecolor=black,linestyle=dashed](5.9,6.75)(10,6.75)
\rput[bl](9.2,6.825){\tiny $y=\frac{\alpha}{2}x_{0}$}

\psline[linecolor=black,linestyle=dashed](5.9,3.25)(10,3.25)
\rput[bl](9.2,2.925){\tiny $y=-\frac{\alpha}{2}x_{0}$}

\psline[linecolor=red,linestyle=solid](5.9,5.875)(10,5.875)
\rput[bl](9.2,5.95){\color{red} \tiny $y=\frac{\alpha}{4}x_{0}$}

\psline[linecolor=red,linestyle=solid](5.9,4.125)(10,4.125)
\rput[bl](9.2,4.15){\color{red} \tiny $y=-\frac{\alpha}{4}x_{0}$}

%\psline[linestyle=dashed](3,6)(3,4)
%\psarc[linecolor=gray]{-}(1.5,5){3}{-41.338}{41.338}
\pscurve[showpoints=false,dotstyle=o](6.4,9.6)(7,9.3)(7.3,9.1)(7.8,8.6)(8,8)(8.2,7.5)(8.5,7.3)(8.7,7.25)(8.75,7.2)(8.78,7.1)(7.1,6)(7.5,5.5)(8,6.3)(8.7,4.7)(9.2,4)(9.3,3.3)(8.7,3.5)(8.4,2.4)(7.3,2.2)

\end{pspicture}
\caption{Bounds away from the $x$-axis}
\label{fig: second/third case}
\end{figure}
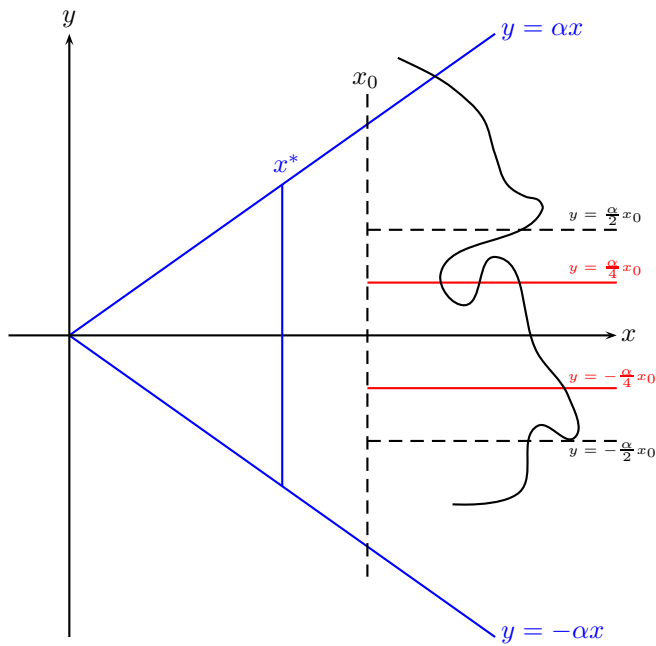

\section{Auxiliary results and proof of Theorem \ref{thm:main}}
%This section contains two parts and the proof of Theorem \ref{thm:main} in the end. 
First, we establish a lower bound for the $x$-component as long as the trajectory stays inside the cone $\mathcal{C}$. Then we formalize what is shown in Figure \ref{fig: second/third case}.\\
%From now on, consider the general case $\sigma_1,\sigma_2 \geq 0$. If either $\sigma_1$ or $\sigma_2$ is equal to zero, then we interpret expressions like $1/\sigma_i$ as infinite. 
 Define
\begin{equation*}
\begin{split}
\overline{\tau}(z)&\coloneqq\inf\{t\geq 0\colon \phi^{(2)}_t(z)\geq \alpha \phi^{(1)}_t(z)\}\text{,}\\
\underline{\tau}(z)&\coloneqq\inf\{t\geq 0\colon  \phi^{(2)}_t(z)\leq -\alpha \phi^{(1)}_t(z)\}\text{,} \\
\tau(z)&\coloneqq \overline{\tau}(z)\wedge\underline{\tau}(z)
\text{.}
\end{split}
\end{equation*}
\subsection{Lower bound}
Note, that  we have a lower bound $\varepsilon >0$ of the following term uniformly for all $(x,y)\in C$
\[
\cos\left(n\arctan\left(\frac{y}{x}\right)\right) \geq \varepsilon >0
\text{.}
\]
Because of $F\in \mathcal{O}(\vert z\vert^{n-1} )\subset\text{o}(\vert z\vert^n )$ as $\vert z \vert \to \infty$, there exists $x^*>0$, such that
\[
\frac{\vert \hat{F_1}(x,y)\vert}{\left(x^2+y^2\right)^{\frac{n}{2}}}  \leq \frac{\varepsilon}{2}
\text{}
\] 
holds for all $x\geq x^*$, $y\in\mathbb{R}$.\\
Fix  $c>0$, $x_0\geq x^*+c$ and $z\in\mathfrak{I}=\{x_0\} \times [-\tan\left(\frac{\pi}{2n}\right) x_0,\tan\left(\frac{\pi}{2n}\right)x_0]$. Then on the event
\[
\{\tau(z)>T\} \cap \{\sup_{t\in[0,T]}\sigma \vert\Wone{t}\vert \leq c\}\text{,}
\]
we have for all $t\in[0,T]\cap D$ ($D$ is the maximal domain on which $X_t$ is defined)
\begin{equation*}
\begin{split}
X_t&=x_0+\int_0^t b_1(X_s,Y_s)+\hat{F_1}(X_s,Y_s)\mathrm{d}s +\sigma\Wone{t}\\
&\geq x_0-c +\int_0^t \left(X_s^2+Y_s^2\right)^{\frac{n}{2}}\left(\cos\left(n\arctan\left(\frac{Y_s}{X_s}\right)\right) - \frac{\vert \hat{F_1}(X_s,Y_s)\vert}{\left(X_s^2+Y_s^2\right)^{\frac{n}{2}}}\right)\mathrm{d}s\\
&\geq  x_0-c +\frac{\varepsilon}{2}\int_0^t X_s^n\mathrm{d}s
\text{.}
\end{split}
\end{equation*}
Applying a (reversed) Gronwall type argument (similar to  \cite[page 83f]{Bihari56}), we see that for all $t\in[0,T]\cap D$ 
\EqSplit{
X_t \geq \frac{x_0-c}{\left(1-\frac{\varepsilon}{2}(n-1)(x_0-c)^{n-1}t\right)^{\frac{1}{n-1}}}
}{eq:xlowerbound}
We define $T\coloneqq \frac{1}{\frac{\varepsilon}{2}(n-1)(x_0-c)^{n-1}}$ which is an upper bound for the explosion time, i.e.\ $(X_t)$ blows up up to time $T$ on the set $\{\tau(z)>T\} \cap \{\sup_{t\in[0,T]}\sigma \vert\Wone{t}\vert \leq c\}$.  Observe that the heuristic ideas remain valid on $\{\sup_{t\in[0,T]}\sigma \vert\Wone{t}\vert \leq c\}$ when replacing $x_0$ by $x_0-c$: if  the event  $\{\sup_{t\in[0,T]}\sigma \vert\Wone{t}\vert \leq c\}$ occurs, then  any trajectory starting in $\mathfrak{I}$ which does not 
leave the cone $\mathcal{C}$ up to $T$ blows up before (or at) time $T$.

\subsection{Bounds away from the $x$-axis}

Throughout the rest of the paper, $c >0$ will be fixed and $x_0>x^*+c$ is a number which will later be sent to infinity. \\
Because of $F\in\mathcal{O}(\vert z \vert^{n-1} )$ there is a $C>0$ such that for all $(x,y)\in\mathbb{R}^2$ with $\vert (x,y)\vert \geq x^*$, where $x^*>0$ is sufficiently large,
\[
\frac{\vert \hat{F_2}(x,y)\vert}{\vert (x,y) \vert^{n-1}}  \leq C
\]
holds true.
Further, we define 
$x_1:=x_0-c$ and $T:=\frac{1}{\frac{\varepsilon}{2}(n-1)x_1^{n-1}}$ as before. Observe that $x_1$ tends to $\infty$ and $T$ tends to $0$ as $x_0 \to \infty$.

\begin{thesis}
\label{lemma:bounds}
For $x^*$ sufficiently large, the following holds. 
Let $(X_t,Y_{t})_{t\in[0,T]}$ solve equation (\ref{eq:cartesian_driving})
with initial condition $(X_0,Y_0 )=z \in  \{x_0 \}\times [-\tan\left(\frac{\pi}{2n}\right) x_0,\tan\left(\frac{\pi}{2n}\right) x_0]$.\\
Define $\nu^+\coloneqq \inf\{t\geq 0\colon Y_t \geq\frac{\alpha}{2}x_1\}$.
Then for all $t\in[\nu^+,\tau(z)]\cap D$, where, again, $D$ is the maximal domain on which $X_t$ is defined, we have
\[
Y_{t}\geq \frac{\alpha}{4}x_1 \quad \text{ on } \quad \lbrace\sup_{t\in [0,T]}\sigma \vert \Wtwo{t} \vert \leq \frac{\alpha}{8}x_1\rbrace\cap \lbrace \inf_{t\in [0,T]\cap D} X_t \geq x_1\rbrace \eqqcolon B
\text{.}
\]
\end{thesis}
\begin{proof}
\newcommand{\tmin}[0]{{(\nu^++t)\wedge\tau}}
Define $\tau\coloneqq\inf\lbrace t>\nu^+\colon Y_t\leq \frac{\alpha}{4}x_1\rbrace \wedge \tau(z)$. We will show $\tau=\tau(z)$ on $B$, which proves the statement.
For $t\geq 0$, such that $\nu^+ + t \in D$ we have
\begin{equation*}
\begin{split}
Y_{\tmin} &=Y_{\nu^+} +\sigma (\Wtwo{\tmin} -\Wtwo{\nu^+})\\
&\hspace{60pt}+\int_{\nu^+}^\tmin b_2(X_s,Y_{s})+\hat{F_2}(X_s,Y_s)\mathrm{d}s\\
&\geq \frac{\alpha}{4}x_1 +\int_{\nu^+}^\tmin \vert (X_s,Y_s) \vert^{n-1} \times \\
&\hspace{60pt}
\left[\sqrt{X_s^2+Y_s^2}\sin\left(n\arctan\left(\frac{Y_s}{X_s}\right)\right)-\frac{\vert \hat{F_2}(X_s,Y_s)\vert}{\vert (X_s,Y_s) \vert^{n-1}}\right]\mathrm{d}s\\
&\geq \frac{\alpha}{4}x_1 +\int_{\nu^+}^\tmin \vert (X_s,Y_s) \vert^{n-1} \times \\
&\hspace{60pt}
\underbrace{ \left[
\sqrt{X_s^2+Y_s^2}\sin\left(n\arctan\left(\frac{Y_s}{X_s}\right)\right)-C
\right]}_{\eqqcolon I}
\mathrm{d}s\\
& \geq \frac{\alpha}{4}x_1
\text{.}
\end{split}
\end{equation*}
We justify the last step by showing $I\geq 0$. First, note that there exists a $b_n >0$, such that for all $z\in [0,b_n)$
\[
\sin(n\arctan(z)) \geq z
\text{.}
\]
Second, define $a_n\coloneqq \sin(n\arctan(b_n))$. Recall, that for $s\in [\nu^+,\tau]$, we have $\alpha X_s \geq Y_s \geq \frac{\alpha}{4}x_1$, which implies
\[
\sin\left(n\arctan\left(\frac{Y_s}{X_s}\right)\right) \geq a_n \wedge \frac{Y_s}{X_s}
\text{.} 
\]
Finally,
\begin{equation*}
I \geq \left( a_n \wedge \frac{Y_s}{X_s}\right)\sqrt{X_s^2+Y_s^2} - C \geq \big(a_nX_s\big) \wedge Y_s - C \geq \big(a_n x_1\big) \wedge \frac{\alpha}{4}x_1 - C 
\end{equation*}
is non-negative if we choose $x^*$ (and therefore also $x_1$) sufficiently large.
\end{proof}

\begin{comment}
If $\nu^+$ is replaced by $\nu^-\coloneqq \inf\{t\geq 0\colon Y_t \leq-\frac{\alpha}{2}x_1\}$ then we obtain in the same way
\[
Y_{t} \leq -\frac{\alpha}{4}x_1
\text{}
\]
for  $t\in[\nu^-,\tau(z)]\cap D$.
\end{comment}
The previous  lemma and  remark are a formal description of what was explained in Section \ref{sec:heuristic}, see also Figure \ref{fig: second/third case}. It will be very useful to show the existence of points which stay inside $\mathcal{C}$ until explosion (Lemma \ref{lemma:inc}).

%The key of the idea of the proof of Theorem \ref{thm:main} is that once we restrict the Brownian motion, we can conduct the above argument for every stopping time $\nu$ with $\nu\leq T$ and every initial condition $y\in[-x_0,x_0]$. Indeed, we use this argument several times, but we only pay for it once in terms of probability mass, that we lose by looking at the event  $\lbrace\sup_{t\in [0,T]}\sigma_2 \vert \Wtwo{t} \vert \leq \frac{\alpha}{8}x_1\rbrace$.\\
Recall that $\tau(z)$ is the exit time of $\mathcal{C}$ for $z\in\mathfrak{I}$.
\begin{thesis}
\label{lemma:inc}
\begin{equation*}
\{\sup_{z\in\mathfrak{I}}\tau(z)> T\}\supset \lbrace\sup_{t\in [0,T]}\sigma \vert \Wtwo{t} \vert \leq \frac{\alpha}{8}x_1\rbrace\cap \lbrace\sup_{t\in[0,T]}\sigma \vert\Wone{t}\vert \leq c\rbrace
\text{.}
\end{equation*}
\end{thesis}
\begin{proof}
Define the random sets
\begin{equation*}
\begin{split}
R &\coloneqq \{z\in\mathfrak{I}\colon \overline{\tau}(z) \leq \underline{\tau}(z)\wedge T \}\text{,} \\
B &\coloneqq \{z\in\mathfrak{I}\colon \underline{\tau}(z) \leq \overline{\tau}(z)\wedge T \}\text{,} \\
G &\coloneqq \mathfrak{I}\setminus \left(R\cup B \right)
\text{.}
\end{split}
\end{equation*}
Note that $R$ and $B$ are disjoint and
\[
\{\sup_{z\in\mathfrak{I}}\tau(z)> T\} = \{ G\neq \emptyset \}
\text{.}
\]
For ease of notation we define
\begin{equation*}
B_1 \coloneqq \lbrace\sup_{t\in[0,T]}\sigma \vert\Wone{t}\vert \leq c\rbrace  \quad B_2\coloneqq \lbrace\sup_{t\in [0,T]}\sigma \vert \Wtwo{t} \vert \leq \frac{\alpha}{8}x_1\rbrace 
\text{.}
\end{equation*}
Let $\omega \in B_1\cap B_2$.\\
Since $\omega \in B_1$ there is a minimal drift in the $x$-component for all trajectories starting in $\mathfrak{I}$ as long as they stay inside 
$\mathcal{C}$. Furthermore, there is a lower bound in the $x$-coordinate for those trajectories, namely $x_1=x_0-c$.\\
Obviously $R(\omega)$ and $B(\omega)$ are not empty since $(x_0,x_0)\in R(\omega)$ and $(x_0,-x_0)\in B(\omega)$.\\
Assume now that $\omega \in \{G=\emptyset\}$ which is equivalent to $\omega \in\{\mathfrak{I}=R\cup B\}$. We show that $R(\omega)$ and $B(\omega)$ are (non-empty) closed subsets of $\mathfrak{I}$, whose disjoint union is equal to the connected set $\mathfrak{I}$, which is a contradiction.\\
Take a converging sequence $z_n\to z$ with $z_n\in R(\omega)$ for all $n\in\mathbb{N}$ and assume that $z\in B(\omega)$. Then, thanks to the continuity of $\phi_t(z,\omega)$ in $(t,z)$, there is a (random) $n\in\mathbb{N}$ such that
\[
\sup_{t\in[0,\tau(z)]}\vert \phi_{t}^{(2)}(z,\omega)-\phi_{t}^{(2)}(z_n,\omega)  \vert \leq \frac{\alpha}{3}x_1
\text{.}
\]
Due to Lemma \ref{lemma:bounds}, we can conclude that $\phi_{t}^{(2)}(z,\omega) $ was never above $\alpha x_1/2 $ before time $\tau(z)$, and 
therefore $\phi_{t}^{(2)}(z_n,\omega)$ was never above $5 \alpha x_1/6 $. Because of $z\in B(\omega)$, there is a time  $\underline{\tau}(z)(\omega)< T$ such that $\phi_{\underline{\tau}(z)(\omega)}^{(2)}(z,\omega)\leq -\alpha x_1 $, which means that $\phi_{\underline{\tau}(z)(\omega)}^{(2)}(z_n,\omega)\leq -2\alpha x_1/3 $. Again, due to Lemma \ref{lemma:bounds}, $z_n$ cannot be in $R(\omega)$. Since this is a contradiction, we have $z\notin B(\omega)$ and therefore $z\in R(\omega)$. Thus, $R(\omega)$ is closed and, by symmetry, so is $B(\omega)$. Therefore the proof of the lemma is complete.
\end{proof}
\subsection{Proof of Theorem \ref{thm:main}}
Note that with the lower bound on the $x$-component (see (\ref{eq:xlowerbound})) we have the following inclusion
\begin{equation*}
\lbrace\sup_{z\in \mathfrak{I}}\sup_{t\leq T}\phi_{t}^{(1)}(z)=\infty \rbrace \supset \{\sup_{z\in\mathfrak{I}}\tau(z)> T\}\cap \{\sup_{t\in[0,T]}\sigma \vert\Wone{t}\vert \leq c\}\eqqcolon A
\text{.}
\end{equation*}
We show that the probability of $A$  already tends to $1$ as $x_0\to\infty$.
\begin{equation}\label{eineGl}
\Prob{A}{}  \geq \mathbb{P}\left( A,\sup_{t\in [0,T]}\sigma \vert \Wtwo{t} \vert \leq \frac{\alpha}{8}x_1\right)
\end{equation}
Lemma \ref{lemma:inc} allows us to omit the event $\{\sup_{z\in\mathfrak{I}}\tau(z)> T\}$, so the right hand side of \eqref{eineGl} equals 
\begin{align*}
&= \mathbb{P}\left(\sup_{t\in[0,T]}\sigma \vert\Wone{t}\vert \leq c,\sup_{t\in [0,T]}\sigma \vert \Wtwo{t} \vert \leq \frac{\alpha}{8}x_1\right)\\
%&= \mathbb{P}\left(\sup_{t\in[0,T]}\sigma_1 \vert\Wone{t}\vert \leq c,\tau(z^+)\vee\tau(z^-)< T,\sup_{t\in [0,T]}\sigma_2 \vert \Wtwo{t} \vert \leq \frac{\alpha}{8}x,\right.\\
%&\hspace{96pt}\left. \inf_{0\leq s\leq \tau(z^+)}\phi_{s}^{(2)}(z^+)\geq 0, \sup_{0\leq s\leq \tau(z^-)}\phi_{s}^{(2)}(z^-)\leq 0\right)\\
%&\geq \mathbb{P}\left(\sup_{t\in[0,T]}\sigma_1 \vert\Wone{t}\vert \leq c,\sup_{t\in [0,T]}\sigma_2 \vert \Wtwo{t} \vert \leq \frac{\alpha}{8}x\right)\\
&\geq 1-\Prob{ \sup_{t\in[0,T]}\sigma \vert\Wone{t}\vert >c}{}-\Prob{ \sup_{t\in [0,T]}\sigma \vert \Wtwo{t} \vert >\frac{\alpha}{8}x_1}{}\\
%&\geq 1\underbrace{-4\Phi\left(-\frac{c}{\sigma_1\sqrt{T}}\right)-4\Phi\left(-\frac{\alpha x}{8\sigma_2\sqrt{T}}\right)}_{\eqqcolon C}
\end{align*}
which converges to 1 as $x_0 \to \infty$ (which implies $x_1 \to \infty$ and $T \to 0$).
%In the last step we used the auxiliary Lemma \ref{lemma:boundBM}. Note
%\[
%\frac{c}{\sigma_1\sqrt{T}} = \frac{c\sqrt{x}\sqrt{1-\alpha^2}}{\sigma_1} \quad\text{ and }\quad \frac{\alpha x}{8\sigma_2\sqrt{T}} =\frac{\alpha x\sqrt{x}\sqrt{1-\alpha^2}}{8\sigma_2}
%\text{,}
%\]
%which means $C\to0$ as $x_0 \to \infty$ (recall $x=x_0-c$). 
This completes the proof.
%\begin{comment}
%In the proof of Lemma \ref{lemma:inc} we used the continuity property of $\phi$ without knowing if it already exploded or not. But, if $E$ denotes the event on which none of the% points $z\in \mathfrak{I}$ explodes up to time $T$, then $E$ and $ B_1 \cap B_2$ are disjoint (recall the notation (\ref{notation})). Further, the proof of our main theorem sho%ws that $\mathbb{P}( B_1 \cap B_2) \geq 1- C$ and thus $\mathbb{P}(E)\leq C$. Note that $C$ converges to zero as $x_0$ goes to infinity.
%\end{comment}
\begin{comment}
We never used any specific properties of the Brownian motions $\Wone{},\Wtwo{}$, apart from the fact that both are processes which start in 0 and have continuous paths. 
Note that in this case the SDE \eqref{eq:cartesian_driving} written in integral form can be solved pathwise for each $\omega \in \Omega$ and the local maximal solutions depend continuously 
upon the initial condition, so all arguments above remain valid in this case. Depending on the nature of the noise, the equation may or may not be complete.
\end{comment}

\bibliographystyle{alpha}
\bibliography{Bibliography}
\Addresses
\end{document}